\documentclass[preprint,12pt]{elsarticle}
\usepackage[T1]{fontenc}
\usepackage[latin9]{inputenc}
\usepackage{amsmath}
\usepackage{amssymb}
\usepackage{color}

\makeatletter

\numberwithin{figure}{section}
\newtheorem{thm}{Theorem}
\newtheorem{lem}[thm]{Lemma}
\newproof{proof}{Proof}
\newtheorem{mydef}{Definition}
\newtheorem{exam}{Example}

\makeatother

\begin{document}

\begin{frontmatter}

\title{Characterization of half-radial matrices}

\author{Iveta Hn\v{e}tynková}\ead{iveta.hnetynkova@mff.cuni.cz}

\author{Petr Tich\'y}\ead{petr.tichy@mff.cuni.cz}

\address{Faculty of Mathematics and Physics, Charles University, \\Sokolovsk\'{a} 83, Prague 8,
Czech Republic}

\journal{Linear Algebra and its Applications}

\begin{abstract}
Numerical radius $r(A)$ is the radius of the smallest ball with the center at zero containing
the field of values of a given square matrix $A$. It is well known that 
$r(A)\leq \|A\| \leq 2r(A)$, where $\| \cdot \|$ is the matrix 2-norm.
Matrices attaining the lower bound are called radial, and have been analyzed
thoroughly. This is not the case for matrices attaining the upper bound where only partial results are available. 
In this paper we consider matrices satisfying $r(A)=\|A\|/2$, and call them half-radial. 
We summarize the existing results and formulate new ones. In particular, 
we investigate their singular value decomposition and algebraic structure,
and provide other necessary and sufficient conditions for a matrix to be half-radial. 
Based on that, we study the extreme case of the attainable 
constant~$2$ in Crouzeix's conjecture. The presented results support 
the conjecture of Greenbaum and Overton, that 
the Crabb-Choi-Crouzeix matrix always plays an important role 
in this extreme case.

\end{abstract}


\end{frontmatter}

\section{Introduction}

Consider the space $\mathbb{C}^{n}$ endowed with the Euclidean vector
norm $\|z\|=\left(z^{*}z\right)^{1/2}$ and the Euclidean inner product
$\left\langle z,w\right\rangle =w^{*}z$, 
where $*$ denotes
the Hermitian transpose. Let $A\in\mathbb{C}^{n\times n}$.
The \emph{field of values} (or numerical range) of $A$ is the set
in the complex plane defined as 
\begin{equation}\label{eq:fieldofval}
W(A)\equiv\{\left\langle Az,z\right\rangle :\ z\in\mathbb{C}^{n},\ \|z\|=1\}.
\end{equation}
Recall that $W(A)$ is a compact convex subset of $\mathbb{C}$ which
contains the spectrum $\sigma(A)$ of $A$; see, e.g., \cite[p.~8]{B:HoJo1994}.
The radius of the smallest ball with the center at zero which contains
the field of values, i.e., 
\[
r(A)\equiv\max_{\|z\|=1}|\left\langle Az,z\right\rangle |=
\max_{\zeta\in W(A)}|\zeta|,
\]
is called the \emph{numerical radius}. It is well-known that 
\begin{equation}
r(A)\leq\|A\|\leq2r(A),\label{eq:basic}
\end{equation}
where $\|\cdot\|$ denotes the matrix 2-norm; see, e.g., \cite[p.~331, problem~21]{B:HoJo1990}.
The lower as well as the upper bound in \eqref{eq:basic} are attainable.
If $\rho(A)$ denotes the \emph{spectral radius} of $A$, $\rho(A)=\max\{|\lambda|:\ \lambda\in\sigma(A)\}$, then it holds that 
\[
r(A)=\|A\|\quad\mbox{if and only if}\quad\rho(A)=\|A\|;
\]
see \cite[p.~44, problem~24]{B:HoJo1994} or \cite{Wi1929}. Such
matrices are called \emph{radial}, and there exist several equivalent
conditions which characterize them; see, e.g., \cite[p.~45, problem~27]{B:HoJo1994}
or \cite{GoldZwa74}. 

Concerning the right inequality in \eqref{eq:basic}, a sufficient condition was formulated
in \cite[p.~11, Theorem 1.3-4]{GuRa1997}: If the range of $A$ and $A^*$ are orthogonal then the upper bound is attained.
It is also known that $\|A\|=2r(A)$ if and only if $A$ has a two-dimensional reducing subspace on which it is the shift \cite[p.11, Theorem~1.3-5]{GuRa1997}. 
Thus the field of values of such $A$ is a disc of specific properties; see \cite[Chapter 25-7]{B:hog2013} for a summary. 
However, to our knowledge a deeper analysis of matrices satisfying $r(A)=\frac12\|A\|$, which we call {\em half-radial}, has not been given yet.

In this paper we fill this gap. We derive several equivalent (necessary and sufficient) conditions characterizing half-radial matrices.
We study in more detail their algebraic structure, 
their left and right singular subspaces corresponding to the maximum singular value, etc.   
We show that $A$ is half-radial if and only if there exists a
special subset of the set of maximizers of the quantity $|\left\langle Az,z\right\rangle|$.
The presented results give a strong indication that half-radial matrices correspond to the extreme case of the attainable constant~$2$ in Crouzeix's conjecture.
Motivated by the conjecture of Greenbaum and Overton in \cite[p.~239]{GrOv2018}, and using Crabb's theorem \cite[Theorem~2]{Cr1971}, 
it is shown that a matrix $A$ reaches the upper bound in Crouzeix's inequality with
the constant $2$ for some monomial if and only if its scalar multiple can be unitarily transformed to a block diagonal form with 
one block being the Crabb-Choi-Crouzeix matrix.

The paper is organized as follows. Section~\ref{sec:basics} gives some basic notation and summarizes well-known results on the field of values. Section~\ref{sec:conditions} is the core part giving characterizations of half-radial 
matrices. Section~\ref{sec:consequency} discusses Crouzeix's conjecture. Finally, Section~\ref{sec:conclusion} concludes the paper. 

\section{Preliminaries}\label{sec:basics}

In this section we introduce some basic notation and prove for completeness the inequality \eqref{eq:basic}.
First, recall that any two vectors $z$ and $w$ satisfy the \emph{polarization
identity}
\begin{eqnarray*}
4\left\langle Az,w\right\rangle  & = & \left\langle A(z+w),(z+w)\right\rangle -\left\langle A(z-w),(z-w)\right\rangle \\
 &  & +\mathbf{i}\left\langle A(z+\mathbf{i}w),(z+\mathbf{i}w)\right\rangle -\mathbf{i}\left\langle A(z-\mathbf{i}w),(z-\mathbf{i}w)\right\rangle ,
\end{eqnarray*}
where $\mathbf{i}$ is the imaginary unit, and the \emph{parallelogram
law}
\[
\left\Vert z+w\right\Vert ^{2}+\left\Vert z-w\right\Vert ^{2}=2\left(\left\Vert z\right\Vert ^{2}+\left\Vert w\right\Vert ^{2}\right).
\]
The definition of numerical radius implies the inequality

\begin{equation}
\left|\left\langle Az,z\right\rangle \right|\leq r(A)\|z\|^{2}.\label{eq:rA}
\end{equation}
Using these tools we can now easily prove the following well-known
result; see, e.g., \cite[p.~331, problem~21]{B:HoJo1990} or \cite[p.~9, Theorem~1.3-1]{GuRa1997}.
\begin{thm}
For $A\in\mathbb{C}^{n\times n}$ it holds that 
\[
r(A)\leq\|A\|\leq2r(A).
\]
\end{thm}

\begin{proof}
The left inequality is trivial. Now using the polarization identity, the inequality \eqref{eq:rA}, and the parallelogram law 
we find out that for any $z$ and $w$ it holds that
\[
4\left|\left\langle Az,w\right\rangle \right|\leq4r(A)\left(\left\Vert z\right\Vert ^{2}+\left\Vert w\right\Vert ^{2}\right).
\]
Consider a unit norm vector $z$ such that $\|A\|=\|Az\|$, and
define $w=\frac{Az}{\|Az\|}$. Then using the previous inequality,
we obtain $4\|A\|\leq8r(A).$
\end{proof}

Denote
\begin{equation}\label{eq:jordblock}
J=\left[\begin{array}{cc}
0 & 1\\
0 & 0
\end{array}\right]\in\mathbb{R}^{2\times2}
\end{equation}
the two-dimensional Jordan block corresponding to the zero eigenvalue (the shift).
The following theorem summarizes known characterizations of matrices
satisfying $\|A\|=2r(A)$;
see \cite[Chapter 25-7]{B:hog2013} for a summary and references.

\begin{thm}\label{thm:knownequiv}
Let $A\in\mathbb{C}^{n\times n}$ be a nonzero matrix.
Then the following conditions are equivalent:
\end{thm}

\begin{enumerate}
\item $\|A\|=2r(A)$, 
\item $A/\|A\|$ is unitarily similar to the matrix 
$\left[ \begin{array}{cc} J & 0\\ 0 & B \end{array} \right]$, where $r(B) \leq 1/2$.
\item $W(A)$ is the disk with the center at zero and the radius $\frac{1}{2}\|A\|.$
\end{enumerate}

Further, consider the singular value decomposition (SVD) of $A$, 
$$ 
A = U \Sigma V^{*}, 
$$
where $U, V \in\mathbb{C}^{n \times n}$ are unitary and $\Sigma$ is diagonal with the singular values of~$A$ on the diagonal.
Denote $\sigma_{max}$
the largest singular value of $A$
satisfying $ \sigma_{max} = \|A\|$. The unit norm left and right singular vectors $u,v$ such that 
\begin{equation}\label{eq:pair}
A v = \sigma_{max} u
\end{equation}
are called a {\em pair of maximum singular vectors} of $A$. Denote
\[
\mathcal{V}_{max}(A)\equiv\left\{ x\in\mathbb{C}^{n}:\ \|Ax\|=\|A\|\|x\|\right\} 
\]
the maximum right singular subspace, i.e., the
span of right singular vectors of $A$ corresponding to $\sigma_{max}$. Similarly,
denote $\mathcal{U}_{max}(A)$ the maximum left singular subspace of $A$.

Recall that any vector $z\in\mathbb{C}^{n}$ can be uniquely decomposed into two
orthogonal components,
\begin{equation}
z=x+y,\qquad x\in\mathcal{R}(A^{*}),\ y\in\mathcal{N}(A),\label{eq:y}
\end{equation}
where $\mathcal{R}(\cdot)$ denotes the \emph{range}, and $\mathcal{N}(\cdot)$
the \emph{null space} of a given matrix. Note that 
here we generally consider the field of complex numbers.

We introduce a definition of matrices we are interested in.

\begin{mydef}
A nonzero matrix $A$ is called half-radial if $\|A\| = 2r(A)$.
\end{mydef}
In the following, we implicitly assume that $A$ is nonzero and $n\geq2$. If $A$ is the zero
matrix or $n=1$, all the statements are trivial.

\section{Necessary and sufficient conditions for half-radial matrices} \label{sec:conditions}

In this section, we derive several new
necessary and sufficient conditions for a nonzero matrix $A$ to be  half-radial.
Taking into account the multiplicity of the maximum singular value of $A$, we then
specify in detail the algebraic structure of half-radial matrices, suggested by Theorem~\ref{thm:knownequiv}.  

Using the decomposition \eqref{eq:y} we define the set 
\begin{equation}\label{def:theta}
\Theta_{A}\equiv\left\{ z\in\mathbb{C}^{n}:\:\|z\|=1,\ r(A)=\left|\left\langle Az,z\right\rangle \right|,\ \left\langle Ax,x\right\rangle =0\right\} ,
\end{equation}
i.e., the set of all unit norm maximizers of $\left|\left\langle Az,z\right\rangle \right|$
such that $\left\langle Ax,x\right\rangle =0.$ 
In the following we prove that a matrix is half-radial if and only if
$\Theta_{A}\neq\{\emptyset\}$, and provide a complete characterization
of the non-empty set $\Theta_{A}$ based on the maximum singular subspaces of $A$. 

\subsection{Basic conditions}

A sufficient condition for $A$ to be half-radial is formulated in the following lemma.
Parts of the proofs of Lemma~\ref{lem:two} and Lemma~\ref{lem:Rao} are inspired by \cite[p.11, Theorem 1.3--4, Theorem 1.3--5]{GuRa1997}.
\begin{lem}
\label{lem:two} If $\Theta_{A}$ is non-empty, then $\|A\|=2r(A)$. Moreover,
for each $z\in\Theta_{A}$ of the form \eqref{eq:y} it holds
that $\left\Vert x\right\Vert =\left\Vert y\right\Vert =\frac{1}{\sqrt{2}}$,
$\|Ax\| = \|A\| \|x\|$ and $\left|\left\langle Ax,y\right\rangle \right|=\|Ax\|\|y\|$.
\end{lem}

\begin{proof}
Let $z$ be a unit norm vector. Considering its decomposition \eqref{eq:y},
it holds that
\begin{equation}
\left\langle Az,z\right\rangle =\left\langle Ax,y\right\rangle +\left\langle Ax,x\right\rangle .\label{eq:Ayy}
\end{equation}
Since $1=\|z\|^{2}=\|x\|^{2}+\|y\|^{2}$ and $0\leq\left(\left\Vert x\right\Vert -\left\Vert y\right\Vert \right)^{2}$,
we find out that 
\[
\left\Vert x\right\Vert \left\Vert y\right\Vert \leq\frac{1}{2}\,,
\]
with the equality if and only if $\left\Vert x\right\Vert =\left\Vert y\right\Vert .$ 
If $\Theta_{A}\neq\{\emptyset\}$, then for any unit norm vector $z\in\Theta_{A}$
we get 
\[
\frac{\left\Vert A\right\Vert }{2}\leq r(A)=\left|\left\langle Az,z\right\rangle \right|=\left|\left\langle Ax,y\right\rangle \right|\leq\|Ax\|\|y\|\leq\left\Vert A\right\Vert \left\Vert x\right\Vert \left\Vert y\right\Vert \leq\frac{\left\Vert A\right\Vert }{2}.
\]
Hence all terms are equal, and therefore $2r(A)=\|A\|$ and $\|x\|=\|y\|=\frac{1}{\sqrt{2}}.$
\end{proof}
In the following lemma we prove the opposite implication. 
\begin{lem}
\label{lem:Rao} If $\|A\|=2r(A),$ then $\Theta_{A}$ is non-empty. Moreover,
for any pair $u, v$ of maximum singular vectors of $A$ it holds that $v\in\mathcal{N}(A^{*}), u\in\mathcal{N}(A), u \perp v$ 
and the vector $z=\frac{1}{\sqrt{2}}\left(u+v\right)$ satisfies $z \in \Theta_{A}$.
\end{lem}

\begin{proof}
Assume without loss of generality that $\|A\| = \sigma_{max}=1$.
Consider now any unit norm maximum right singular vector $v\in\mathcal{V}_{max}(A)$ and denote
$u=Av$ the corresponding maximum left singular vector. Then
\[
\|u\|=\|Av\|=\|A\|\|v\|=\|v\| =1. 
\]
Moreover, for any $\theta\in\mathbb{R}$ we obtain 
\[
r(A)=r\left(e^{\mathbf{i}\theta}A\right)=\frac{1}{2}.
\]
Then, since $e^{\mathbf{i}\theta}A+e^{-\mathbf{i}\theta}A^{*}$ is Hermitian, it holds that
\[
\left\Vert e^{\mathbf{i}\theta}A+e^{-\mathbf{i}\theta}A^{*}\right\Vert = r\left(e^{\mathbf{i}\theta}A+e^{-\mathbf{i}\theta}A^{*}\right) 
\leq r\left(e^{\mathbf{i}\theta}A\right)+r\left(e^{-\mathbf{i}\theta}A^{*}\right) = 1\,
\]
and
\[
\left\Vert e^{\mathbf{i}\theta}Au+e^{-\mathbf{i}\theta}A^{*}u\right\Vert ^{2}\leq1,\ \left\Vert e^{\mathbf{i}\theta}Av+e^{-\mathbf{i}\theta}A^{*}v\right\Vert ^{2}\leq1\,.
\]
For the norm on the left we get 
\begin{eqnarray*}
\left\Vert e^{\mathbf{i}\theta}Au+e^{-\mathbf{i}\theta}A^{*}u\right\Vert ^{2} 
& = & e^{2\mathbf{i}\theta} \left\langle A^2u,u\right\rangle 
+ e^{-2\mathbf{i}\theta} \left\langle (A^*)^2u,u\right\rangle 
+ \left\langle Au,Au\right\rangle + \left\langle A^*u,A^*u\right\rangle \\
& = &  2\mathrm{Re}\left(e^{2\mathbf{i}\theta}\left\langle A^{2}u,u\right\rangle \right) +
\left\Vert Au\right\Vert ^{2} + \left\Vert A^*u\right\Vert ^{2}.
\end{eqnarray*}
Recalling that $\|v\|= \|A^*u\|=1$, we obtain
\[
\left\Vert Au\right\Vert ^{2}\leq - 2\mathrm{Re}\left(e^{2\mathbf{i}\theta}\left\langle A^{2}u,u\right\rangle \right).
\]
Similarly from $\|u\|= \|Av\|=1$, it can be proved that
\[
 \left\Vert A^{*}v\right\Vert ^{2}\leq-2\mathrm{Re}\left(e^{2\mathbf{i}\theta}\left\langle A^{2}v,v\right\rangle \right).
\]
Since the inequalities hold for any $\theta\in\mathbb{R}$, we get $\|A^*v\|=\|Au\|=0$, i.e., 
$v\in\mathcal{N}(A^{*})$ and $u\in\mathcal{N}(A)$. Hence, $v\perp\mathcal{R}(A)$,
$u\perp\mathcal{R}(A^{*})$ and the maximum left and right singular vectors $u, v$ are orthogonal. Furthermore,
\[
\left\langle Av,v\right\rangle = \left\langle Au,v\right\rangle =\left\langle Au,u\right\rangle =0.
\]

Clearly, the unit norm vector
$$
z=\frac{1}{\sqrt{2}}\left(u+v\right)
$$
is in the form \eqref{eq:y} with $x= \frac{1}{\sqrt{2}} \, v, y= \frac{1}{\sqrt{2}} \, u$.
Moreover, $z$ is a maximizer by
\begin{eqnarray}
\left\langle Az,z\right\rangle   & = & \frac{1}{2}\left\langle A(u+v),(u+v)\right\rangle \nonumber \\
 & = & \frac{1}{2}\left(\left\langle Av,v\right\rangle +\left\langle Av,u\right\rangle +\left\langle Au,v\right\rangle +\left\langle Au,u\right\rangle \right)\nonumber\\ 
 & = & \frac{1}{2}\left\langle Av,u\right\rangle \nonumber
\end{eqnarray}
yielding 
$$
|\left\langle Az,z\right\rangle | =\frac{1}{2}=r(A).
$$
Since also $\left\langle Ax,x\right\rangle = \frac{1}{2}\left\langle Av,v\right\rangle=0$, it holds that $z\in\Theta_{A}$.
\end{proof}

Recall that for any pair $u, v$ of singular vectors of $A$ it holds that $v\in\mathcal{R}(A^{*}), u\in\mathcal{R}(A)$. Consequently, the previous lemmas imply the following result characterizing half-radial matrices by the properties of their maximum right singular subspace. 
\begin{lem}
\label{lem:four}$\|A\|=2r(A)$ if and only if for any unit norm vector $v\in\mathcal{V}_{max}(A)$
it holds that 
\begin{equation}\label{eq:strange-1}
\begin{split}
v\in\mathcal{R}(A^{*})\cap\mathcal{N}(A^{*}),\qquad Av\in\mathcal{R}(A)\cap\mathcal{N}(A), \\
\mbox{and} \quad z = \frac{1}{\sqrt{2}} \left(v + \frac{Av}{\|A\|}\right) \quad \mbox{maximizes} \quad |\left\langle Az,z\right\rangle|.
\end{split}
\end{equation}
Equivalently, $\|A\|=2r(A)$ if and only if there exists a unit norm
$v\in\mathcal{V}_{max}(A)$ such that the condition \eqref{eq:strange-1}
is satisfied.
\end{lem}

\begin{proof}
Let $\|A\|=2r(A)$. For any $v\in\mathcal{V}_{max}(A)$, $\|v\|=1$,
$$
u = \frac{Av}{\|A\|} \quad \mbox{and} \quad v 
$$ 
is a pair of maximum singular vectors of $A$. Thus from Lemma~\ref{lem:Rao} it follows that 
$z$ defined in \eqref{eq:strange-1} is a unit norm maximizer of $|\left\langle Az,z\right\rangle|$.
Furthermore, by the same lemma $v \in\mathcal{N}(A^{*})$ and $Av = \|A\| u \in\mathcal{N}(A)$. 
Moreover, since $v$ and $u$ are right and left singular vectors, we also
have $v\in\mathcal{R}(A^{*})$, $Av\in\mathcal{R}(A)$ yelding \eqref{eq:strange-1}. 

On the other hand, if there is a unit norm right singular vector $v\in\mathcal{V}_{max}(A)$ 
such that the condition \eqref{eq:strange-1} is satisfied, then the vector
\[
z=\frac{1}{\sqrt{2}}\left(v+u\right), \qquad \mbox{where} \qquad u=\frac{Av}{\|A\|},
\]
is a unit norm maximizer of $|\left\langle Az,z\right\rangle|$.
Moreover, $v$ is the component of $z$ in $\mathcal{R}(A^{*})$ and 
$\left\langle Av,v\right\rangle =0$.
From \eqref{def:theta}, $z$ is an element of $\Theta_{A}$. Hence, $\Theta_{A}$ is non-empty and by Lemma~\ref{lem:two} we conclude that $\|A\|=2r(A)$. 
\end{proof}

The previous lemma implies several SVD properties of half-radial matrices,
which we summarize below. 
 
\begin{lem}\label{lem:subspaces} 
 $A\in\mathbb{C}^{n\times n}$ be half-radial, i.e., $\|A\|=2r(A)$. Then:
\begin{enumerate}
\item $\mathcal{V}_{max}(A)\subseteq \mathcal{R}(A^{*})\cap\mathcal{N}(A^*), 
\quad \mathcal{U}_{max}(A)\subseteq\mathcal{R}(A)\cap \mathcal{N}(A),$
\label{cond:range}
\item $\mathcal{V}_{max}(A) \perp \mathcal{U}_{max}(A)$,
\label{cond:perp}
\item the multiplicity $m$ of $\sigma_{max}$ is not greater than $n/2$, 
\item the multiplicity of the zero singular value is at least $m$.
\end{enumerate}
\end{lem}

\begin{proof}
The first property follows from the condition \eqref{eq:strange-1} in Lemma~\ref{lem:four} 
and the fact that $\mathcal{V}_{max}(A) \subseteq \mathcal{R}(A^{*})$ and $\mathcal{U}_{max}(A) \subseteq \mathcal{R}(A)$. 
Combining $\mathcal{V}_{max}(A) \subseteq \mathcal{R}(A^{*})$ with $\mathcal{U}_{max}(A) \subseteq \mathcal{N}(A)$
yields the orthogonality of maximum singular subspaces. Furthermore,
since $\mathcal{V}_{max}(A), \mathcal{U}_{max}(A) \subseteq \mathbb{C}^{n}$ and
$\mathcal{V}_{max}(A) \perp \mathcal{U}_{max}(A)$, we directly get the restriction on the multiplicity of $\sigma_{max}$. 
Finally, the relation $\mathcal{V}_{max}(A)\subseteq \mathcal{N}(A^*)$ implies that $A$ is singular with
the multiplicity of the zero singular value being equal to or greater than the multiplicity of  $\sigma_{max}$.
\end{proof}

It is worth to note that neither the condition~\ref{cond:perp}
nor the stronger condition~\ref{cond:range} in 
Lemma~\ref{lem:subspaces} are sufficient to ensure that a given matrix 
$A$ is half-radial. This can be illustrated by the following examples.

\begin{exam}\label{ex:1}
Consider a half-radial matrix $A$ such that the multiplicity of $\sigma_{max}$
(and thus also the dimension of maximum singular subspaces) is maximum possible, i.e., $n/2$. Then
\[
\mathcal{V}_{max}(A)  \oplus \mathcal{U}_{max}(A) =  \mathbb{C}^{n},
\]
and $\mathcal{U}_{max}(A)=\mathcal{N}(A)$, $\mathcal{V}_{max}(A)=\mathcal{N}(A^*)$. Thus 
$A$ has only two singular values $\sigma_{max}$ and $0$ both with the multiplicity~$n/2$.
Consequently, an~SVD of half-radial $A$ has in this case the form
\begin{equation}\label{eq:maxradial}
A = U \left[ \begin{array}{cc} \sigma_{max}I & { 0} \\ { 0} & { 0}  \end{array} \right] V^*, \qquad I \in \mathbb{R}^{n/2\times n/2}.
\end{equation}
The unitary matrices above can be chosen as $U = [U_{max}| V_{max}]$, $V = [V_{max}| U_{max}]$, where the columns of
$U_{max}$ and $V_{max}$ form an orthonormal basis of $\mathcal{U}_{max}(A)$ and $\mathcal{V}_{max}(A)$, respectively, 
and $Ue_i, Ve_i$ is for $i=1, \dots, n/2$ a pair of maximum singular vectors of $A$. 

However, assuming only that $A$ has orthogonal maximum right and left singular subspaces, both of the dimension $n/2$, 
we can just conclude that an~SVD of $A$ has in general the form
\begin{equation}\label{eq:maxgeneral}
A = U \left[ \begin{array}{cc} \sigma_{max}I & { 0} \\ { 0} & { \tilde \Sigma}  \end{array} \right] V^*, \qquad I \in \mathbb{R}^{n/2\times n/2},
\end{equation}
with $\tilde \Sigma \in \mathbb{R}^{n/2\times n/2}$ having singular values smaller than $\sigma_{\max}$ on the diagonal.
\end{exam}

The previous example also shows that if the multiplicity of $\sigma_{max}$ is $n/2$, then assuming 
only that the condition \ref{cond:range} from Lemma~\ref{lem:subspaces} holds, the matrix $A$ is half-radial. 
However, this is not true in general as we illustrate by the next example.

\begin{exam}\label{ex:2}
Consider the matrix given by its SVD
\[
A=
\left[\begin{array}{ccc}
0 & 0 & 0\\
0 & 0.9 & 0\\
1 & 0 & 0
\end{array}\right]
=\left[\begin{array}{ccc}
0 & 0 & 1\\
0 & 1 & 0\\
1 & 0 & 0
\end{array}\right]
\left[\begin{array}{ccc}
1 & 0 & 0\\
0 & 0.9 & 0\\
0 & 0 & 0
\end{array}\right]\left[\begin{array}{ccc}
1 & 0 & 0\\
0 & 1 & 0\\
0 & 0 & 1
\end{array}\right].
\]
Obviously, $\mathcal{U}_{max}(A)=span\{e_3\}, \mathcal{V}_{max}(A)=span\{e_1\}$.
Since $Ae_3 = 0$ and $A^*e_1 = 0$, we obtain $\mathcal{U}_{max}(A) \subseteq\mathcal{N}(A)$
and $\mathcal{V}_{max}(A)\subseteq \mathcal{N}(A^*)$. The vector $e_2$ 
is a (unit norm) eigenvector of $A$. Consequently,
$$
r(A) \geq |\left\langle Ae_2,e_2\right\rangle| = |\left\langle 0.9 \, e_2,e_2\right\rangle| = 0.9.
$$
Since $\|A\|=1$, we see that $A$ is not half-radial, even though the condition 
\ref{cond:range} from Lemma~\ref{lem:subspaces} holds.

\end{exam}

\subsection{Set of maximizers $\Theta_{A}$}

Now we study the set of maximizers $\Theta_{A}$.
Lemma~\ref{lem:two} implies that for any $z\in\Theta_{A}$ its component $x$ in $\mathcal{R}(A^*)$ satisfies $x \in \mathcal{V}_{max}(A)$. On the other hand, Lemma~\ref{lem:Rao} says that for half-radial matrices a maximizer from $\Theta_{A}$ can be constructed as a linear combination of the vectors $u,v$ representing a (necessarily orthogonal) pair of maximum singular vectors of $A$.
We are ready to prove that the whole set $\Theta_{A}$ is formed by linear combinations of such vectors, particularly we define the set
\[
\Omega_{A}\equiv\left\{ \frac{1}{\sqrt{2}}\left(e^{\boldsymbol{i}\alpha}v+e^{\boldsymbol{i}\beta}u\right):\ v\in\mathcal{V}_{max}(A),\ \|v\|=1,\ u=\frac{Av}{\|A\|},\ \alpha,\ \beta\in\mathbb{R}\right\}
\]
and get the following lemma.

\begin{lem}
\label{lem:five} $\Theta_{A}$ is either empty or $\Theta_{A}=\Omega_{A}.$
\end{lem}

\begin{proof}
Let $\Theta_{A}$ be non-empty. 
Consider for simplicity a unit norm vector $z\in\Theta_{A}$ in the scaled form \eqref{eq:y},
$$
z=\frac{1}{\sqrt{2}}\left(x+y\right),
$$
i.e., $x\perp y$, $\|x\|=\|y\|=1$, $\left\langle Ax,x\right\rangle =0.$
We first show that $z\in\Omega_{A}$.
From Lemma~\ref{lem:two} we know that $x$ 
satisfies $\|Ax\|=\|A\|$, and $Ax$ and $y$ are linearly dependent. Hence,
\[
Ax=\gamma y,\qquad\gamma=\left\langle Ax,y\right\rangle 
\]
with 
$$
|\gamma|= |\left\langle Ax,y\right\rangle|=
\|Ax\| \|y\| = \|A\| \|x\| \|y\| = \|A\| = \sigma_{max}.
$$ 
Now we can rotate the vector $y$ in order to get the pair of maximum singular vectors.
Define the unit norm vector $\tilde{y}$ by 
\[
\tilde{y}\equiv e^{-\boldsymbol{i}\beta}y,\qquad e^{-\boldsymbol{i}\beta}\equiv\frac{\left\langle Ax,y\right\rangle }{\|A\|},
\]
then $Ax=\gamma y=\left(\|A\|e^{-\boldsymbol{i}\beta}\right)\left(e^{\boldsymbol{i}\beta}\tilde{y}\right)=\|A\|\tilde{y}.$ Thus $x, \tilde{y}$ is a pair of maximum singular vectors of $A$. 
Hence, it holds that 
\[
z=\frac{1}{\sqrt{2}}\left(x+y\right)=\frac{1}{\sqrt{2}}\left(x+e^{\boldsymbol{i}\beta}\tilde{y}\right),\qquad x\in\mathcal{V}_{max}(A),\quad\tilde{y}=Ax/\|A\|,
\]
giving $z\in\Omega_{A}$. Therefore, $\Theta_{A}\subseteq\Omega_{A}.$ 

Now we prove the opposite inclusion $\Omega_{A}\subseteq\Theta_{A}$.
Since $\Theta_{A}\neq\{\emptyset\}$, it holds that $\|A\|=2r(A)$ and any vector $z$ of the form
\[
z=\frac{1}{\sqrt{2}}\left(v+u\right),\quad v\in\mathcal{V}_{max}(A),\ \|v\|=1,\ u=Av/\|A\|
\]
is an element of $\Theta_{A}$, see Lemma~\ref{lem:Rao}. Consider now a unit norm vector $\tilde{z}\in\Omega_{A}$ of the form 
\[
\tilde{z}=\frac{1}{\sqrt{2}}\left(e^{\boldsymbol{i}\alpha}v+e^{\boldsymbol{i}\beta}u\right).
\]
We know that $\|u\|=\|v\|=1$, $u\perp v$, and $e^{\boldsymbol{i}\alpha}v\in\mathcal{R}(A^{*})$,
$e^{\boldsymbol{i}\beta}u\in\mathcal{N}(A)$. Moreover, 
\[
\left|\left\langle A\tilde{z},\tilde{z}\right\rangle \right|=\left|\frac{1}{2}\left\langle A\left(e^{\boldsymbol{i}\alpha}v\right),\left(e^{\boldsymbol{i}\beta}u\right)\right\rangle \right|=\frac{1}{2}\left|\left\langle Av,u\right\rangle \right|=\frac{\|A\|}{2}=r(A),
\]
so that $\tilde{z}$ is a maximizer. Since also $\left\langle Ae^{\boldsymbol{i}\alpha}v,e^{\boldsymbol{i}\alpha}v\right\rangle =\left\langle Av,v\right\rangle =0$,
we get $\tilde{z}\in\Theta_{A}$.
\end{proof}

Note that in general, the set of maximizers $\Theta_{A}$ need not represent the set of all maximizers of $|\langle Az,z\rangle|$. 
We illustrate that by an example.

\begin{exam}
Consider the matrix given by its SVD
\[
A=\left[\begin{array}{ccc}
0 & 1 & 0\\
0 & 0 & 0\\
0 & 0 & \frac{1}{2}
\end{array}\right]=\left[\begin{array}{ccc}
1 & 0 & 0\\
0 & 0 & 1\\
0 & 1 & 0
\end{array}\right]\left[\begin{array}{ccc}
1 & 0 & 0\\
0 & \frac{1}{2} & 0\\
0 & 0 & 0
\end{array}\right]\left[\begin{array}{ccc}
0 & 0 & -1\\
1 & 0 & 0\\
0 & 1 & 0
\end{array}\right].
\]
It holds that $\|A\|=1$ and $r(A)=\frac{1}{2}.$ From Lemma~\ref{lem:five},
we see that $\Theta_{A}$ is generated by the maximum right singular
vector $e_{2}$ and the corresponding maximum left singular vector
$e_{1}.$ But obviously, the right singular vector $e_{3}$
corresponding to the singular value $\frac{1}{2}$ is not in $\Theta_{A}$
and also represents a maximizer, since 
\[
e_{3}^{T}Ae_{3}=\frac{1}{2}=\frac{1}{2}\|A\|.
\]
\end{exam}

We have seen that for half-radial matrices the vectors of the form 
\[
\frac{\left(e^{\boldsymbol{i}\alpha}v+e^{\boldsymbol{i}\beta}u\right)}{\|e^{\boldsymbol{i}\alpha}v+e^{\boldsymbol{i}\beta}u\|},
\quad \mbox{with} \quad
\ v\in\mathcal{V}_{max}(A), \ \|v\| =1,\ u=\frac{Av}{\|A\|},\ \alpha,\ \beta\in\mathbb{R},
\]
are maximizers of $|\langle Az,z\rangle|$. This is also true for Hermitian, and, more generally, for normal and radial 
matrices. It remains an open question, whether it is possible to find
larger classes of matrices such that the maximizers are of the form above.

\subsection{Algebraic structure of half-radial matrices}

In \cite[p.11, Theorem 1.3--5]{GuRa1997} it was shown 
that a matrix satisfying $\|A\|=2r(A)$ has a two-dimensional reducing subspace on which it is the shift $J$; see the condition $2$ in Theorem~\ref{thm:knownequiv}. 
Recalling that the multiplicity of the zero singular value of a half-radial matrix is at least the multiplicity of $\sigma_{max}$, see Lemma~\ref{lem:subspaces}, it is possible
to analogously extract as many matrices $J$ from $A$ as is the multiplicity of $\sigma_{max}$. 
Denote by $A \otimes B$ a Kronecker product of $A$ and $B$, and by $A \oplus B$ a block--diagonal matrix with the blocks $A, B$ on the diagonal. 
The following lemma gives a full characterization of half-radial matrices from the point
of view of their algebraic structure.

\begin{lem}
\label{lem:six}Let $A\in\mathbb{C}^{n\times n}$ be a nonzero matrix
such that $\dim\mathcal{V}_{max}(A)=m$. It holds that $\|A\|=2r(A)$ if
and only if $A$ is unitarily similar to the matrix 
\[
(\|A\|I_{m}\otimes J)\oplus B,
\]
where $B$ is a matrix satisfying $\|B\|<\|A\|$ and $r(B)\leq\frac{1}{2}\|A\|.$ 
\end{lem}

\begin{proof}
Assume that $\|A\|=2r(A)$. Let $u_i,v_i, i=1, \dots, m$ be pairs of maximum
singular vectors of $A$ such that  $\{v_{1},\dots,v_{m}\}$ is an orthonormal
basis of $\mathcal{V}_{max}(A)$ and $\{u_{1},\dots,u_{m}\}$ is an  
orthonormal basis of $\mathcal{U}_{max}(A)$. Define 
\[
Q\equiv[u_{1},v_{1},u_{2},v_{2},\dots,u_{m},v_{m},P],
\]
where $P$ is any matrix such that $Q$ is unitary (the vectors 
$v_{1},\dots,v_{m},u_{1},\dots,u_{m}$ are mutually orthogonal by Lemma~\ref{lem:subspaces}). Then 
\[
Q^{*}AQ=(\|A\|I_{m}\otimes J)\oplus B,\qquad B=P^{*}AP.
\]
Since the maximum singular value of $A$ corresponds to the block $\|A\|I_{m}\otimes J$, the maximum singular value of $B$ 
is smaller than $\sigma_{max}=\|A\|$. Using a technique analogous to \cite[p.11, Theorem 1.3--5]{GuRa1997}, we also get $r(B)\leq\frac{1}{2}\|A\|.$ 

On the other hand, let $A$ be unitarily similar to the matrix 
\[
(\|A\|I_{m}\otimes J)\oplus B
\]
with $r(B)\leq\frac{1}{2}\|A\|$. 
Recall that for any square matrices $C$ and $D$ it holds that $W(C\oplus D)=\mathrm{cvx}(W(C)\cup W(D))$; see, 
e.g., \cite[p.~12]{B:HoJo1994}. 
Since $W(\|A\| J)$ is the disk with the center at zero and the radius $\frac{1}{2}\|A\|$, and $r(B)\leq\frac{1}{2}\|A\|$, it holds that
$$
 r(A)=\max\{r(\|A\| J),r(B)\}=\frac{1}{2}\|A\|,
$$
Hence, $A$ is half-radial.
\end{proof}

It is well known that $W(J)$ is the disk with the center at zero and 
the radius $\frac{1}{2}$. Since $\|J\|=1$, the matrix $J$ is half-radial.
Thus in words, Lemma~\ref{lem:six} shows that a matrix $A$ is half-radial if and only if it is unitarily similar to a block diagonal matrix with one block being a half-radial matrix with maximum multiplicity of $\sigma_{max}$ (see \eqref{eq:maxradial}) and the other block having smaller norm and numerical radius.

We have discussed previously that orthogonality of maximum singular subspaces does not ensure
that a given matrix $A$ is half-radial (see Example~\ref{ex:1}). Note that assuming only 
$\mathcal{U}_{max}(A) \perp \mathcal{V}_{max}(A)$, we can also get
the block structure of~$A=(\|A\|I_{m}\otimes J)\oplus B$ with $\|B\|<\|A\|$.
However, if $W(B)$ is not contained in $W(\|A\|J)$, then 
$$
r(A)=r(B)> \frac{1}{2}\|A\|,
$$ 
and $A$ is not half-radial. This can be illustrated on Example~\ref{ex:2}, where 
clearly $B = 0.9$ and thus $\|B\| = r(B) = 0.9$ while $\|A\|=1$. \\

Note that Lemma~\ref{lem:six} immediately implies the condition 3 from Theorem~\ref{thm:knownequiv}. 
In particular, for half-radial matrices it holds that
\[
W(A)=\mathrm{cvx}(W(\|A\|J)\cup W(B))=\|A\| W(J).
\]
Moreover, since $W\left(\frac{A^*+A}{2}\right)=\mathrm{Re}\,(W(A))$, for a half-radial $A$ we always have $r\left(\frac{A^*+A}{2}\right)=r(A)$, i.e., $\|A^*+A\| =\|A\|$. \\

To summarize the main results, the following theorem extends Theorem~\ref{thm:knownequiv} by giving several necessary and sufficient conditions characterizing half-radial matrices. 

\begin{thm}
Let $A\in\mathbb{C}^{n\times n}$ be a nonzero matrix.
Then the following conditions are equivalent:
\end{thm}

\begin{enumerate}
\item $\|A\|=2r(A)$, 
\item $\Theta_{A}$ is non-empty,
\item $\Theta_{A}=\Omega_{A}$,
\item $\forall v\in\mathcal{V}_{max}(A)$, $\|v\| = 1$, it holds that 
\begin{equation}\label{eq:strange-1a}
\begin{split}
v\in\mathcal{R}(A^{*})\cap\mathcal{N}(A^{*}),\qquad Av\in\mathcal{R}(A)\cap\mathcal{N}(A), \\
\mbox{and} \quad z = \frac{1}{\sqrt{2}} \left(v + \frac{Av}{\|A\|}\right) \quad \mbox{maximizes} \quad |\left\langle Az,z\right\rangle|,
\end{split}
\end{equation}
\item $\exists v\in\mathcal{V}_{max}(A)$, $\|v\|=1$, such that \eqref{eq:strange-1a} holds,
\item $A$ is unitarily similar to the matrix
\[
(\|A\|I_{m}\otimes J)\oplus B,
\]
where $m$ is the dimension of $\mathcal{V}_{max}(A)$, $\|B\|<\|A\|$ and
$r(B)\leq\frac{1}{2}\|A\|.$ 
\item $W(A)$ is the disk with the center at zero and the radius $\frac{1}{2}\|A\|.$
\end{enumerate}

\section{Consequences on Crouzeix's conjecture}\label{sec:consequency}

In \cite{Cr2007}, Crouzeix proved that for any square matrix $A$
and any polynomial~$p$
\begin{equation}
\|p(A)\|\leq c \max_{\zeta\in W(A)}|p(\zeta)|,\label{eq:crouzeix1}
\end{equation}
where $c=11.08$. Later in \cite{Cr2004} he conjectured that the constant could be replaced by $c=2$. Recently, it has been proven in \cite{CrPa2017} that the inequality 
\eqref{eq:crouzeix1} holds with the constant $c=1+\sqrt{2}$. Still, it remains an open question whether {\it Crouzeix's inequality} 
\begin{equation}
\|p(A)\|\leq2\max_{\zeta\in W(A)}|p(\zeta)|\label{eq:crouzeix}
\end{equation}
is satisfied for any square matrix $A$ and any polynomial $p$.

Based on the results in \cite{OkAn1975}, it has been shown in \cite{BaCrDe2006}
that if the field of values $W(A)$ is a disk, then \eqref{eq:crouzeix}
holds. Thus, for half-radial matrices we conclude the following.

\begin{lem}\label{lem:cr1}\
Half-radial matrices satisfy Crouzeix's inequality~\eqref{eq:crouzeix}. Moreover, the bound with the constant $2$ 
is attained for the polynomial $p(\zeta)=\zeta$. 
\end{lem}

\begin{proof}
Let $A$ be half-radial. Then $W(A)$ is a disk giving the inequality \eqref{eq:crouzeix}.
Furthermore, for $p(\zeta)=\zeta$ we have 
\[
\|p(A)\|=\|A\|=2r(A)=2\max_{\zeta\in W(A)}|\zeta|=2\max_{\zeta\in W(A)}|p(\zeta)|.
\]
\end{proof}

There are other matrices for which the inequality \eqref{eq:crouzeix} holds and  the bound is attainable for
some polynomial $p$. In particular, for the Crabb-Choi-Crouzeix matrix of the form,
\begin{equation}\label{eqn:CC}
C_1=2J,\quad C_n=\left[\begin{array}{cccccc}
0 & \sqrt{2}\\
 & 0 & 1\\
 &  & 0 & \ddots\\
 &  &  & 0 & 1\\
 &  &  &  & 0 & \sqrt{2}\\
 &  &  &  &  & 0
\end{array}\right]\in\mathbb{R}^{(n+1)\times (n+1)},\ n > 1,
\end{equation}
where $J$ is defined in \eqref{eq:jordblock},
the polynomial is $p(\zeta)=\zeta^{n}$; see \cite{Ch2013} and \cite{GrOv2018}. Note that 
$C_n^{n}=2e_{1}e_{n+1}^{T}$,
so that $C_n^{n}$ is unitarily similar (via a permutation matrix) to the matrix $2J\oplus0$.
Hence
\[
	\|C_n^n\|=2r(C_n^n),
\]
and thus $C_n^{n}$ is half-radial. More generally, we can prove the following result.
\begin{lem}\label{lem:one}
Let an integer $k\geq 1$ be given. It holds that  
\begin{equation}\label{eqn:Ak}
	\|p(A)\|=2\max_{\zeta\in W(A)}|p(\zeta)|
\end{equation}
for $p(\zeta)=\zeta^k$ if and only if $A^{k}$ is half-radial
and $r(A^k) = r(A)^k$.
\end{lem}

\begin{proof} Using $r(A^{k})\leq r(A)^{k}$ we obtain
\begin{equation}\label{eqn:cink}
\|A^k\| \leq 2r(A^k) \leq 2r(A)^k = 2 \max_{\zeta\in W(A)}|\zeta^{k}|.
\end{equation}
If \eqref{eqn:Ak} holds for $p(\zeta)=\zeta^k$, then all terms in 
\eqref{eqn:Ak} are equal, and, therefore, 
$A^{k}$ is half-radial and $r(A^k) = r(A)^k$.
On the other hand, if $A^{k}$ is half-radial and $r(A^k) = r(A)^k$, then 
all terms in  \eqref{eqn:cink} are equal, and, 
\eqref{eqn:Ak} holds.
\end{proof}

In \cite[p.~239]{GrOv2018}, Greenbaum and Overton conjectured that if the equality 
\eqref{eqn:Ak} holds for $A\in\mathbb{C}^{(n+1)\times (n+1)}$ and 
the polynomial $p(\zeta)=\zeta^{n}$, then $\alpha A$ (for some nonzero scalar $\alpha$) is unitarily similar to the Crabb-Choi-Crouzeix matrix~\eqref{eqn:CC}.
This result follows from Crabb's theorem \cite[Theorem~2]{Cr1971}, as it is shown
for completeness in Lemma~\ref{lem:crabmatrix} below.\footnote{Abbas Salemi pointed out Crabb's theorem to Greenbaum and Overton, who then conveyed the information to us.
}

\begin{thm}{(Crabb \cite[Theorem~2]{Cr1971})}\label{thm:crabb}
Let $A$ be a bounded linear operator on a Hilbert space $H$, and let $v\in H$. Suppose
that $r(A)=\|v\|=1$ and that $\|A^{k}v\|=2$ for some integer $k$.
Then $A^{k+1}v=0$, $\|A^{i}v\|=\sqrt{2}$, $i=1,2,\dots,k-1$, the
elements $v,Av,\dots,A^{k}v$ are mutually orthogonal, and their linear
span is a reducing subspace of $A$.
\end{thm}
Note that since $r(A)=1$
in Theorem~\ref{thm:crabb}, the condition $\|A^{k}v\|=2$ implies that 
$\|A^{k}\|=2$.
Using Crabb's theorem, the following result follows easily. 

\begin{lem}\label{lem:crabmatrix}
Let $A\in\mathbb{C}^{(n+1)\times (n+1)}$ be a nonzero matrix
scaled such that $r(A)=1$. It holds that 
\begin{equation}
\|A^{n}\|=2\max_{\zeta\in W(A)}|\zeta^{n}|\label{eq:test}
\end{equation}
if and only if $A$ is unitarily similar to the Crabb-Choi-Crouzeix matrix $C_n$ defined in \eqref{eqn:CC}.
\end{lem}

\begin{proof}
Suppose that \eqref{eq:test} holds. Then, using Lemma~\ref{lem:one}, $A^{n}$
is half-radial with $\|A^{n}\|=2r(A^{n})=2r(A)^{n}=2$.
Let $v$ be a maximum right singular vector of $A^{n}$, i.e. $\|A^{n}v\| = 2$.
Using Crabb's Theorem~\ref{thm:crabb} for $k = n$, the matrix 
\begin{equation} \label{eq:crabbQ}
Q\equiv\left[\frac{A^{n}v}{2},\frac{A^{n-1}v}{\sqrt{2}},\dots,\frac{Av}{\sqrt{2}},v\right]
\end{equation} 
is unitary. Furthermore, 
\[
AQ=\left[0,\frac{A^{n}v}{\sqrt{2}},\frac{A^{n-1}v}{\sqrt{2}},\dots,\frac{A^{2}v}{\sqrt{2}},Av\right]=QC_{n}.
\]
Therefore, $A$ is unitarily similar to $C_{n}$.

Since $C^n_n$ is half-radial, $C_{n}$ satisfies \eqref{eq:test} giving directly the other implication.
\end{proof}

Finally, we are ready to prove a generalization of Lemma~\ref{lem:crabmatrix}
for the $k$th power of $A$, $1\leq k\leq n$.

\begin{thm}\label{thm:final}
Let $A\in\mathbb{C}^{(n+1)\times (n+1)}$ be a nonzero matrix.
It holds that 
\begin{equation}
\|A^{k}\|=2\max_{\zeta\in W(A)}|\zeta^{k}|\label{eq:test-1}
\end{equation}
for some $1\leq k\leq n$ if and only if $A$ is unitarily similar
to the matrix
\begin{equation}
r(A)\left(C_{k}\oplus B\right), \label{eq:final}
\end{equation}
where $C_k$ is the Crabb-Choi-Crouzeix matrix defined in \eqref{eqn:CC} and $B$ is a square matrix of the size $n-k$ satisfying $r(B)\leq1$
and $\|B^{k}\|\leq2$.
\end{thm}

\begin{proof}
If $A$ is unitarily similar to the matrix \eqref{eq:final},
then 
\[
\|A^{k}\|=r(A)^{k}\max(\|C_{k}^{k}\|,\|B^{k}\|)=2r(A)^{k}=2\max_{\zeta\in W(A)}|\zeta^{k}|.
\]

On the other hand, if the condition \eqref{eq:test-1} is satisfied, define the matrix 
$\tilde{A} = r(A)^{-1}A$. 
Following a construction similar as in \eqref{eq:crabbQ} for $\tilde{A}$ and the index~$k$, 
we get the matrix $Q_k \in\mathbb{C}^{(n+1) \times (k+1)}$ with orthonormal
columns such that $\tilde{A}Q_k=Q_kC_{k}.$ Using Crabb's Theorem~\ref{thm:crabb}, columns
of $Q_k$ span a reducing subspace of $\tilde{A}$. Therefore, there exists
a square matrix $B \in\mathbb{C}^{(n-k) \times (n-k)}$ and a matrix 
$\tilde{Q} \in \mathbb{C}^{(n+1) \times (n-k)}$ such that $[Q_k,\tilde{Q}]$
is unitary, and 
\begin{equation}
r(A)^{-1}A[Q_k,\tilde{Q}]=[Q_k,\tilde{Q}]\left(C_{k}\oplus B\right). \label{eq:tmp1}
\end{equation}
Moreover, combining Lemma~\ref{lem:one} and \eqref{eq:tmp1} we obtain 
\[
2r(A)^{k}=\|A^{k}\|=r(A)^{k}\max(\|C_{k}^{k}\|,\|B^{k}\|),
\]
i.e., $\max(\|C_{k}^{k}\|,\|B^{k}\|)=2$ and $\|B^{k}\|\leq2$.
Finally, from \eqref{eq:tmp1} it follows that $1=\max(r(C_{k}),r(B))$, and
therefore $r(B)\leq1.$
\end{proof}

\section{Conclusions}\label{sec:conclusion}

In this paper, we have investigated half-radial matrices, and provided several equivalent characterizations based on their
properties. Our results reveal that half-radial matrices represent a very special class of matrices. 
This leads us to the question, whether the constant $2$ in the inequality 
$||A||\leq 2 r(A)$ can be improved for some larger classes of nonnormal matrices. For example, since the constant $2$ corresponds to 
matrices with $0 \in W(A)$, one can think about an improvement of the bound if $0\notin W(A)$.

Finally, our results suggest that half-radial matrices are related 
to the case when the upper bound in Crouzeix's inequality can be attained for some polynomial.
This is supported by one of the conjectures of Greenbaum and Overton \cite[p.~242]{GrOv2018} predicting that the upper bound in Crouzeix's inequality can be attained only if $p$ is a monomial. In particular, 
if this conjecture is true, then Theorem~\ref{thm:final} characterizes all matrices for which the upper bound in Crouzeix's inequality
can be attained. A deeper analysis of this phenomenon needs further research which is, however, beyond the scope of this paper.

\section{Acknowledgment}
This research was supported by the Grant Agency of the Czech Republic under the grant no. 17-04150J.
We thank Marie Kub\'{i}nov\'{a} for careful reading of an earlier version of this manuscript 
and for pointing us to Example~2. We also thank two anonymous referees for their very helpful comments. We are very grateful to Anne Greenbaum 
for comments that led us to significant improvements of the chapter devoted to Crouzeix's conjecture.

\end{document}